\documentclass{amsproc}

\allowdisplaybreaks

\usepackage{amsmath,amssymb}
\usepackage[mathcal]{euscript}
\usepackage{amsthm}
\usepackage{amsfonts}
\usepackage{amsmath,amssymb}
\usepackage{indentfirst,latexsym,bm,amsthm,graphicx,colortbl}
\usepackage{fancyhdr}
\usepackage{times}
\usepackage{amsmath,amsfonts,amssymb,amsthm}
\usepackage{epsfig}
\usepackage[all]{xy}
\usepackage{amssymb,supertabular,a4}
\usepackage{bm}
\usepackage[mathcal]{euscript}
\usepackage{amsthm}

\newtheorem{theorem}{\textbf{Theorem}}[section]
\newtheorem{lemm}[theorem]{Lemma}
\newtheorem{prop}[theorem]{\textbf{Proposition}}

\theoremstyle{definition}
\newtheorem{defi}[theorem]{\textbf{Definition}}

\theoremstyle{remark}
\newtheorem{remark}[theorem]{\textbf{Remark}}
\theoremstyle{claim}

\numberwithin{equation}{section}

\usepackage{ragged2e}
\usepackage{leftidx}
\usepackage{amscd}
\usepackage{multicol}
\usepackage{amsbsy}
\usepackage{epsfig}
\usepackage{tikz}
\usepackage{cite}
\usepackage{indentfirst}
\usepackage{multirow}
\usepackage{tabularx}
\usepackage{mathrsfs}
\usepackage{amsfonts}
\usepackage{}
\usepackage{amsmath}
\usepackage{wasysym,amssymb,eufrak,indentfirst,graphicx,cite,amsthm,color}
\usepackage{array}
\begin{document}

\baselineskip 17pt

\title[A class of non-weight modules over the Schr\"{o}dinger-Virasoro algebras]
{A class of non-weight modules over the Schr\"{o}dinger-Virasoro algebras}
\author[Wang]{Yumei Wang}
\address{Department of Mathematics, Shanghai University,
Shanghai 200444, China} \email{yumeiwang@shu.edu.cn}

\author[Zhang]{Honglian Zhang$^\star$}
\address{Department of Mathematics, Shanghai University,
Shanghai 200444, China} \email{hlzhangmath@shu.edu.cn}

\thanks{$^\star$ H. Zhang, Corresponding Author}


\begin{abstract}
We construct and classify the free $U(\mathbb{C}L_0\oplus \mathbb{C}M_0\oplus\mathbb{C}Y_0)$-modules of rank $1$ over the Schr\"{o}dinger-Virasoro algebra $\mathfrak{sv}(s)$ for $s=0$.
Moreover, we show that the class of free $U(\mathbb{C}L_0\oplus\mathbb{C}M_0)$-modules of rank $1$ over the Schr\"{o}dinger-Virasoro algebra $\mathfrak{sv}(s)$ for $s=\frac{1}{2}$ is nonexistent.
\end{abstract}

\date{}
\maketitle

\section{Introduction}
In the 1960s, physicist Virasoro gave an important infinite dimensional Lie algebra--the Virasoro algebra
with a basis $\{L_{n},C|n \in \mathbb{Z}\}$ and the following relations
$$[L_{n},L_{m}]=(m-n)L_{m+n}+\delta_{m+n}\frac{n^3-n}{12}C,\ \  m,n\in\mathbb{Z},$$
$$[L_{n},C]=0, \ \ n\in\mathbb{Z},$$
which is a universal central extension of the Witt algebra.
The Virasoro algebra is closely related to string theory\cite{MS} and conformal field theory\cite{KF}.
Additionally, the representation theory of the Virasoro algebra is widely applied in many branches of mathematics and physics, such as vertex algebras \cite{R,FZ} and quantum physics \cite{GO}.
Furthermore, the generalizations of the Virasoro algebra are extensively studied, such as the Schr\"{o}dinger-Virasoro algebras \cite{LS,CHS,WY}, the Virasoro-like algebras \cite{GLW,JL}, the affine Virasoro algebras \cite{GHL} and so on.

The Schr\"{o}dinger-Virasoro algebra is one of the natural generalizations  of the Virasoro algebra, which was introduced by M. Henkel in \cite{H}
during his study of the free Schr\"{o}dinger equations in non-equilibrium statistical physics. The Schr\"{o}dinger-Virasoro algebra $\mathfrak{sv}(s)$ for $s=0$ or $\frac{1}{2}$ (c.f. \cite{H}, \cite{HU}) is an infinite-dimensional Lie algebra over $\mathbb{C}$ with a basis
$\{L_{n},M_{n},Y_{p}|n \in \mathbb{Z}, p\in \mathbb{Z}+s\}$
and satisfying the following relations
\begin{eqnarray*}
&&[L_{n},L_{m}]=(m-n)L_{m+n},\\
&&[L_{n},Y_{p}]=(p-\frac{n}{2})Y_{p+n},\\
&&[L_{n},M_{m}]=mM_{m+n},\\
&&[Y_{p},Y_{q}]=(q-p)M_{p+q},\\
&&[Y_{p},M_{m}]=[M_{n},M_{m}]=0,
\end{eqnarray*}
where $m, n\in\mathbb{Z}$ and $p,q\in \mathbb{Z}+s$.
The Schr\"{o}dinger-Virasoro algebras play an important role in many fields of mathematics and physics. Furthermore, there were a number of works on the Schr\"{o}dinger-Virasoro algebras and their representations theory (see \cite{LS,CHS,WY,ZT1,ZT2} ect.).
Actually, the weight modules of the Schr\"{o}dinger-Virasoro algebras have been studied extensively.
For instance, Li and Su studied the weight modules with finite-dimensional weight spaces over the Schr\"{o}dinger-Virasoro algebras in \cite{LS}.

On the other hand, it is well known that the theory of non-weight modules has been extensively studied in the past few years.
In \cite{BM}, Batra and Mazorchuk gave a general set-up for studying Whittaker modules,
from which Nilsson considered to construct families of simple $\mathfrak{sl}_{n+1}$-modules.
In \cite{Ni}, Nilsson determined and classified the free $U(\mathfrak{h})$-modules of rank $1$ over $\mathfrak{sl}_{n+1}$, in which $\mathfrak{h}$ is the standard Cartan subalgebra of $\mathfrak{sl}_{n+1}$.
Furthermore, the idea of Nilsson in \cite{Ni} has been generalized and applied into many infinite dimensional algebras.
The non-weight modules which are free of rank $1$ over the Kac-Moody algebras and the classical Lie superalgebras were studied in \cite{CTZ} and \cite{CZ} respectively.
Moreover, such modules for the Virasoro algebra and its related algebras were investigated in \cite{CC} and \cite{HCS}.
In addition, Chen and Guo constructed the free modules of rank $1$ and determined the simplicity over the Heisenberg-Virasoro algebra and the $W(2,2)$ algebra in \cite{CG}.

Our goal of the present paper is to focus on the non-weight modules over the Schr\"{o}dinger-Virasoro algebra $\mathfrak{sv}(s)$ for $s=0$ or $\frac{1}{2}$.
We construct a class of free $U(\mathbb{C}L_0\oplus \mathbb{C}M_0\oplus\mathbb{C}Y_0)$-modules of rank $1$ over the algebra $\mathfrak{sv}(0)$, denoted by $\Phi(\lambda,\alpha)$ for $\lambda\in\mathbb{C}^*$ and $\alpha\in\mathbb{C}$.
Moreover, we classify all such modules and obtain the main result of the present paper: Let $\mathfrak{sv}(0)$ be the Schr\"{o}dinger-Virasoro algebra $\mathfrak{sv}(s)$ for $s=0$, if $M$ is a free $U(\mathbb{C}L_0\oplus \mathbb{C}M_0\oplus\mathbb{C}Y_0)$-module of rank $1$ over the algebra $\mathfrak{sv}(0)$, then there exists some $\lambda\in\mathbb{C}^*$ and $\alpha\in\mathbb{C}$ such that ${M}\cong\Phi(\lambda,\alpha).$
In addition, we prove that the class of free $U(\mathbb{C}L_0\oplus\mathbb{C}M_0)$-modules of rank $1$ over the algebra $\mathfrak{sv}(\frac{1}{2})$ is nonexistent.

The paper is organized as follows. In Section 2, we demonstrate the definition and properties of the algebra $\mathfrak{sv}(0)$. Besides, we construct a class of free $U(\mathbb{C}L_0\oplus \mathbb{C}M_0\oplus\mathbb{C}Y_0)$-modules of rank $1$ over the algebra $\mathfrak{sv}(0)$. Section 3 is aimed to classify the free $U(\mathbb{C}L_0\oplus \mathbb{C}M_0\oplus\mathbb{C}Y_0)$-modules of rank 1 over the algebra $\mathfrak{sv}(0)$.
In Section 4, we mainly prove the nonexistence of free $U(\mathbb{C}L_0\oplus \mathbb{C}M_0)$-modules of rank $1$ over the algebra $\mathfrak{sv}(\frac{1}{2})$.

\section{\textbf{Preliminaries}}
In this section, we list some basic notations and useful results for our purpose.
Throughout the paper, $\mathbb{N}$, $\mathbb{C}$, $\mathbb{C}^*$, $\mathbb{Z}$ and $\mathbb{Z}^*$ stand for the set of all natural numbers, complex numbers, nonzero complex numbers, integers and nonzero integers, respectively.
Note that we focus on the Schr\"{o}dinger-Virasoro algebra $\mathfrak{sv}(s)$ for $s=0$ in Section 2 and Section 3.

The Schr\"{o}dinger-Virasoro algebra $\mathfrak{sv}(0)$ has a $\mathbb{C}$-basis $\{L_{n},Y_{n},M_{n}|n \in \mathbb{Z}\}$
with the following relations
\begin{eqnarray}
&&[L_{n},L_{m}]=(m-n)L_{m+n},\label{15}\\
&&[L_{n},Y_{m}]=(m-\frac{n}{2})Y_{m+n},\label{16}\\
&&[L_{n},M_{m}]=mM_{m+n},\label{17}\\
&&[Y_{n},Y_{m}]=(m-n)M_{m+n},\label{18}\\
&&[Y_{n},M_{m}]=[M_{n},M_{m}]=0,\label{19}
\end{eqnarray}
where $m\in\mathbb{Z}, n\in\mathbb{Z}$.

From the above relations (\ref{15})-(\ref{19}), we obtain some important formulas that will be needed later.
\begin{prop}\label{4}
For $i\in\mathbb{Z},n\in\mathbb{Z}$ and $i\geq 1$, the following formulas hold:
\begin{enumerate}
  \item $M_{n}L_{0}^{i}=(L_{0}-n)^{i}M_{n}$.\label{proML}
  \item $M_{n}M_{0}^{i}=M_{0}^{i}M_{n}$.\label{proMM}
  \item $M_{n}Y_{0}^{i}=Y_{0}^{i}M_{n}$.\label{proMY}
  \item $Y_{n}L_{0}^{i}=(L_{0}-n)^{i}Y_{n}$.\label{proYL}
  \item $Y_{n}M_{0}^{i}=M_{0}^{i}Y_{n}$.\label{proYM}
  \item $Y_{n}Y_{0}^{i}=Y_{0}^{i}Y_{n}-n iY_{0}^{i-1}M_{n}$.\label{proYY}
  \item $L_{n}L_{0}^{i}=(L_{0}-n)^{i}L_{n}$.\label{proLL}
  \item $L_{n}M_{0}^{i}=M_{0}^{i}L_{n}$.\label{proLM}
  \item $L_{n}Y_{0}^{i}=Y_{0}^{i}L_{n}-\frac{n i}{2}Y_{0}^{i-1}Y_{n}+\frac{n^{2}i(i-1)}{4}Y_{0}^{i-2}M_{n}$.\label{proLY}
\end{enumerate}
\end{prop}
\begin{proof}
Here we only check formula (\ref{proYY}) and formula (\ref{proLY}), others can be easily shown by induction on $i$.
For formula (\ref{proYY}), it is obvious for $i=1$ according to $Y_{n}Y_{0}=Y_{0}Y_{n}+[Y_{n},Y_{0}]=Y_{0}Y_{n}-nM_{n}$.
Now assume that formula (\ref{proYY}) holds for $i=k-1$, that is, $$Y_{n}Y_{0}^{k-1}=Y_{0}^{k-1}Y_{n}-(k-1)nY_{0}^{k-2}M_{n}.$$
For $i=k$, by the inductive assumption we have
\begin{eqnarray*}
\begin{aligned}
Y_{n}Y_{0}^{k}&=Y_{n}Y_{0}^{k-1}Y_{0}\\
                &=Y_{0}^{k-1}Y_{n}Y_{0}-(k-1)nY_{0}^{k-2}M_{n}Y_{0}\\
                &=Y_{0}^{k}Y_{n}-nY_{0}^{k-1}M_{n}-(k-1)nY_{0}^{k-1}M_{n}\\
                &=Y_{0}^{k}Y_{n}-knY_{0}^{k-1}M_{n},
\end{aligned}
\end{eqnarray*}
which implies that formula (\ref{proYY}) holds.

For formula (\ref{proLY}), we first have $L_{n}Y_{0}=Y_{0}L_{n}+[L_{n},Y_{0}]=Y_{0}L_{n}-\frac{n}{2}Y_{n}$, formula (\ref{proLY}) holds for $i=1$.
Next suppose that formula (\ref{proLY}) holds for $i=k-1$, one has $$L_{n}Y_{0}^{k-1}=Y_{0}^{k-1}L_{n}-\frac{(k-1)n}{2}Y_{0}^{k-2}Y_{n}+\frac{(k-1)(k-2)n^{2}}{4}Y_{0}^{k-3}M_{n}.$$
For $i=k$, we immediately obtain that
\begin{eqnarray*}
\begin{aligned}
L_{n}Y_{0}^{k}&=L_{n}Y_{0}^{k-1}Y_{0}\\
              &=Y_{0}^{k-1}L_{n}Y_{0}-\frac{(k-1)n}{2}Y_{0}^{k-2}Y_{n}Y_{0}+\frac{(k-1)(k-2)n^{2}}{4}Y_{0}^{k-3}M_{n}Y_{0}\\
              &=Y_{0}^{k}L_{n}-\frac{n}{2}Y_{0}^{k-1}Y_{n}-\frac{(k-1)n}{2}Y_{0}^{k-1}Y_{n}+\frac{(k-1)n^2}{2}Y_{0}^{k-2}M_{n}\\
              &\quad+\frac{(k-1)(k-2)n^{2}}{4}Y_{0}^{k-2}M_{n}\\
              &=Y_{0}^{k}L_{n}-\frac{nk}{2}Y_{0}^{k-1}Y_{n}+\frac{k(k-1)n^{2}}{4}Y_{0}^{k-2}M_{n}.
\end{aligned}
\end{eqnarray*}

Therefore formula (\ref{proLY}) holds.

\end{proof}

\begin{remark}
Note that these formulas in Proposition \ref{4} also hold in the case of $i=0$.
\end{remark}

\begin{defi}\label{5}
For $\lambda\in\mathbb{C}^*$ and $\alpha\in\mathbb{C}$, define the action of a basis of the algebra $\mathfrak{sv}(0)$ on $\Phi(\lambda,\alpha):=\mathbb{C}[s,t,v]$, in which $\mathbb{C}[s,t,v]$ is the polynomial algebra in three indeterminates $s,t$ and $v$, as follows:
\begin{eqnarray}
L_m\cdot f(s,t,v)&=&\lambda^m\left((s+m\alpha)f(s-m,t,v)-\frac{m}{2}v\frac{\partial f(s-m,t,v)}{\partial v}\right.\nonumber\\
                 &&\left.\ \ \ \ \quad +\frac{m^2}{4}t\frac{\partial^2f(s-m,t,v)}{\partial v^2}\right),\label{12}\\
M_m\cdot f(s,t,v)&=&\lambda^mtf(s-m,t,v),\label{13}\\
Y_m\cdot f(s,t,v)&=&\lambda^m\left(vf(s-m,t,v)-mt\frac{\partial f(s-m,t,v)}{\partial v}\right),\label{14}
\end{eqnarray}
where $m\in\mathbb{Z}$ and $f(s,t,v)\in\mathbb{C}[s,t,v]$.
\end{defi}

\begin{prop}
For $\lambda\in\mathbb{C}^*$ and $\alpha\in\mathbb{C}$, $\Phi(\lambda,\alpha)$ is a class of free modules of rank $1$ over the algebra $\mathfrak{sv}(0)$ under the action defined by (\ref{12})-(\ref{14}).
\end{prop}
\begin{proof}
For $m\in\mathbb{Z}$ and $n\in\mathbb{Z}$, according to (\ref{12}), one has the following
\begin{eqnarray*}\label{11}
\begin{aligned}
&L_n\cdot L_m\cdot f(s,t,v)\\
=&L_n\cdot \bigg[\lambda^m\left((s+m\alpha)f(s-m,t,v)-\frac{m}{2}v\frac{\partial f(s-m,t,v)}{\partial v}
+\frac{m^2}{4}t\frac{\partial^2f(s-m,t,v)}{\partial v^2}\right)\bigg]\\
=&\ \lambda^{m+n}\ \bigg\{(s+n\alpha)(s-n+m\alpha)\ f(s-m-n,t,v)+\Big[-\frac{n}{2}v(s-n+m\alpha)\\
&\quad\quad\quad-\frac{m}{2}(s+n\alpha)v+\frac{mn}{4}v\Big]\frac{\partial f(s-m-n,t,v)}{\partial v}+\Big[\frac{n^2}{4}t(s-n+m\alpha)\\
&\quad\quad\quad+\frac{mn}{4}v^2-\frac{mn^2}{4}t+\frac{m^2}{4}(s+n\alpha)t \Big]\frac{\partial^2f(s-m-n,t,v)}{\partial v^2}\\
&\quad\quad\quad+\Big[-\frac{mn^2}{8}vt-\frac{m^2n}{8}vt\Big]\frac{\partial^3f(s-m-n,t,v)}{\partial v^3}\\
&\quad\quad\quad+\frac{m^2n^2}{16}t^2\frac{\partial^4f(s-m-n,t,v)}{\partial v^4}\bigg\},
\end{aligned}
\end{eqnarray*}
which implies that
\begin{eqnarray*}
\begin{aligned}
&L_n\cdot L_m\cdot f(s,t,v)-L_m\cdot L_n\cdot f(s,t,v)\\
=&\lambda^{m+n}\bigg\{(m-n)\Big[\big(s+(m+n)\alpha\big)f(s-m-n,t,v)-\frac{(m+n)}{2}v\frac{\partial f(s-m-n,t,v)}{\partial v}\\
 &\quad\quad\quad+\frac{(m+n)^2}{4}t\frac{\partial^2f(s-m-n,t,v)}{\partial v^2}\Big]\bigg\}\\
=&(m-n)L_{m+n}\cdot f(s,t,v)\\
=&[L_n,L_m]\cdot f(s,t,v).
\end{aligned}
\end{eqnarray*}

By (\ref{12}) and (\ref{14}), one easily gets
\begin{eqnarray}\label{LY}
&&\quad L_n\cdot Y_m\cdot f(s,t,v)\nonumber\\
&&=L_n\cdot \bigg[\lambda^m\left(vf(s-m,t,v)-mt\frac{\partial f(s-m,t,v)}{\partial v}\right)\bigg]\nonumber\\
&&=\lambda^{m+n}\bigg\{\left((s+n\alpha)v-\frac{n}{2}v\right)f(s-m-n,t,v)+\left(-\frac{n}{2}v^2
+\frac{n^2}{2}t- m(s+n\alpha)t\right)\nonumber\\
&&\quad\quad\quad\frac{\partial f(s-m-n,t,v)}{\partial v}+\left(\frac{n^2}{4}vt
+\frac{mn}{2}vt\right)\frac{\partial^2 f(s-m-n,t,v)}{\partial v^2}\nonumber\\
&&\quad\quad\quad-\frac{mn^2}{4}t^2\frac{\partial^3f(s-m-n,t,v)}{\partial v^3}\bigg\}.
\end{eqnarray}
Similarly, it immediately holds that
\begin{eqnarray}\label{YL}
&&\quad Y_m\cdot L_n\cdot f(s,t,v)\nonumber\\
&&=Y_m\cdot \bigg[\lambda^n\left((s+n\alpha)f(s-n,t,v)-\frac{n}{2}v\frac{\partial f(s-n,t,v)}{\partial v}+\frac{n^2}{4}t\frac{\partial^2f(s-n,t,v)}{\partial v^2}\right)\bigg]\nonumber\\
&&=\lambda^{m+n}\bigg\{v(s-m+n\alpha)f(s-m-n,t,v)+\left(-mt(s-m+n\alpha)-\frac{n}{2}v^2+\frac{n}{2}mt\right)\nonumber\\
&&\quad\quad\quad\frac{\partial f(s-m-n,t,v)}{\partial v}+\left(\frac{n}{2}mtv+\frac{n^2}{4}vt\right)\frac{\partial^2 f(s-m-n,t,v)}{\partial v^2}\nonumber\\
&&\quad\quad\quad-\frac{n^2}{4}mt^2\frac{\partial^3f(s-m-n,t,v)}{\partial v^3}\bigg\}.
\end{eqnarray}
Subtracting (\ref{YL}) from (\ref{LY}) gives rise to
\begin{eqnarray*}
\begin{aligned}
&L_n\cdot Y_m\cdot f(s,t,v)-Y_m\cdot L_n\cdot f(s,t,v)\\
=&\lambda^{m+n}\bigg\{(m-\frac{n}{2})\bigg[v f(s-m-n,t,v)-(m+n)t\frac{\partial f(s-m-n,t,v)}{\partial v}\bigg]\bigg\}\\
=&(m-\frac{n}{2})Y_{m+n}\cdot f(s,t,v)\\
=&[L_n,Y_m]\cdot f(s,t,v).
\end{aligned}
\end{eqnarray*}

Using (\ref{12}) and (\ref{13}), one obtains
\begin{eqnarray}\label{LM}
&&L_n\cdot M_m\cdot f(s,t,v)\nonumber\\
&&=L_n\cdot\Big[\lambda^m tf(s-m,t,v)\Big]\nonumber\\
&&=\lambda^{m+n}\bigg\{(s+n\alpha)tf(s-m-n,t,v)-\frac{n}{2}vt\frac{\partial f(s-m-n,t,v)}{\partial v}\nonumber\\
&&\quad\quad\quad\quad+\frac{n^2}{4}t^2\frac{\partial^2f(s-m-n,t,v)}{\partial v^2}\bigg\}.
\end{eqnarray}
Similarly, it is easy to see that
\begin{eqnarray}\label{ML}
&&\quad M_m\cdot L_n\cdot f(s,t,v)\nonumber\\
&&=M_m\cdot \bigg[\lambda^n\left((s+n\alpha)f(s-n,t,v)-\frac{n}{2}v\frac{\partial f(s-n,t,v)}{\partial v}\right.\nonumber\\
&&\quad\quad\quad\quad\quad\left.+\frac{n^2}{4}t\frac{\partial^2f(s-n,t,v)}{\partial v^2}\right)\bigg]\nonumber\\
&&=\lambda^{m+n}\bigg\{t(s-m+n\alpha)f(s-m-n,t,v)-\frac{n}{2}tv\frac{\partial f(s-m-n,t,v)}{\partial v}\nonumber\\
&&\quad\quad\quad\quad+\frac{n^2}{4}t^2\frac{\partial^2f(s-m-n,t,v)}{\partial v^2}\bigg\}.
\end{eqnarray}
Considering (\ref{LM}) and (\ref{ML}), we have
\begin{eqnarray*}
\begin{aligned}
&L_n\cdot M_m\cdot f(s,t,v)-M_m\cdot L_n\cdot f(s,t,v)\\
=&\lambda^{m+n}mtf(s-m-n,t,v)\\
=&mM_{m+n}\cdot f(s,t,v)\\
=&[L_n,M_m]\cdot f(s,t,v).
\end{aligned}
\end{eqnarray*}

It follows from (\ref{14})
\begin{eqnarray*}
\begin{aligned}
&Y_n\cdot Y_m\cdot f(s,t,v)\\
=&Y_n\cdot\bigg[\lambda^m\left( vf(s-m,t,v)-mt\frac{\partial f(s-m,t,v)}{\partial v}\right)\bigg]\\
=&\lambda^{m+n}\bigg\{\left(v^2-mt\right)f(s-m-n,t,v)+ mnt\frac{\partial ^2 f(s-m,t,v)}{\partial v^2}\bigg\},
\end{aligned}
\end{eqnarray*}
which implies that
\begin{eqnarray*}
\begin{aligned}
&Y_n\cdot Y_m\cdot f(s,t,v)-Y_m\cdot Y_n\cdot f(s,t,v)\\
=&\lambda^{m+n}\left(v^2-mt-v^2+nt\right)f(s-m-n,t,v)\\
=&(n-m)M_{m+n}\cdot f(s,t,v)\\
=&[Y_n,Y_m]\cdot f(s,t,v).
\end{aligned}
\end{eqnarray*}

Finally, due to (\ref{14}) and (\ref{13}), it is straightforward to verify that
\begin{eqnarray*}
\begin{aligned}
& Y_n\cdot M_m\cdot f(s,t,v)-M_m\cdot Y_n\cdot f(s,t,v)\\
=&Y_n\cdot\Big[\lambda^m tf(s-m,t,v) \Big]-M_m\cdot\bigg[\lambda^n\left(vf(s-n,t,v)-nt\frac{\partial f(s-n,t,v)}{\partial v}\right)\bigg]\\
=&0=[Y_n,M_m]\cdot f(s,t,v)
\end{aligned}
\end{eqnarray*}
and
\begin{eqnarray*}
\begin{aligned}
&M_n\cdot M_m\cdot f(s,t,v)-M_m\cdot M_n\cdot f(s,t,v)\\
=&M_n\cdot\left(\lambda^mtf(s-m,t,v)\right)-M_m\cdot\left(\lambda^ntf(s-n,t,v)\right)\\
=&0=[M_n,M_m]\cdot f(s,t,v).
\end{aligned}
\end{eqnarray*}
This completes the proof.
\end{proof}

\begin{remark}
For $\lambda\in\mathbb{C}^*$ and $\alpha\in\mathbb{C}$, $\Phi(\lambda,\alpha)$ is reducible over the Schr\"{o}dinger-Virasoro algebra $\mathfrak{sv}(s)$ for $s=0$ as a consequence of Definition \ref{5}. In fact, it is effortless to get that $t^i\mathbb{C}[s,t,v]$ is a submodule of $\Phi(\lambda,\alpha)$ for $i\in\mathbb{Z}$.
\end{remark}

\section{\textbf{Modules over $\mathfrak{sv}(0)$}}
In this section, we classify the free $U(\mathbb{C}L_0\oplus \mathbb{C}M_0\oplus\mathbb{C}Y_0)$-modules of rank $1$ over the algebra $\mathfrak{sv}(0)$. Indeed, we get the following main result.

\begin{theorem}\label{th}
Let $\mathfrak{sv}(0)$ be the Schr\"{o}dinger-Virasoro algebra $\mathfrak{sv}(s)$ for $s=0$, if there exists a free $U(\mathbb{C}L_0\oplus \mathbb{C}M_0\oplus\mathbb{C}Y_0)$-module of rank $1$ over the algebra $\mathfrak{sv}(0)$, denoted by $M$, then $${M}\cong\Phi(\lambda,\alpha)$$
for some $\lambda\in\mathbb{C}^*$ and $\alpha\in\mathbb{C}$.
\end{theorem}

In order to prove Theorem \ref{th}, we first show several important lemmas.

\begin{lemm}\label{qd1}
Let $M$ be a free $U(\mathbb{C}L_0\oplus \mathbb{C}M_0\oplus\mathbb{C}Y_0)$-module of rank $1$ over the algebra $\mathfrak{sv}(0)$.
For $m\in\mathbb{Z}$, assume that
\begin{eqnarray*}
\begin{aligned}
&L_m\cdot1=g_m(L_0,M_0,Y_0),\\
&M_m\cdot1=a_m(L_0,M_0,Y_0),\\
&Y_m\cdot1=p_m(L_0,M_0,Y_0),
\end{aligned}
\end{eqnarray*}
in which $g_m(L_0,M_0,Y_0),a_m(L_0,M_0,Y_0)$, $p_m(L_0,M_0,Y_0)\in M$,
then $g_m(L_0,M_0,Y_0)$, $a_m(L_0,M_0,Y_0)$ and $p_m(L_0,M_0,Y_0)$ completely determine the action of $L_m, M_m$ and $Y_m$ on $M$.
\end{lemm}
\begin{proof}
Now take any $u(L_0,M_0,Y_0)=\sum\limits_{i,j,k\geq0}a_{i,j,k}L_{0}^{i}M_{0}^{j}Y_{0}^{k}\in M$ and according to formulas (\ref{proLL})-(\ref{proLY}) in Proposition  \ref{4}, one has
\begin{eqnarray*}\label{1}
\begin{aligned}
&L_m\cdot u(L_0,M_0,Y_0)=L_m\cdot \sum\limits_{i,j,k\geq0}a_{i,j,k} L_{0}^{i}M_{0}^{j}Y_{0}^{k} \\
=& \sum\limits_{i,j,k\geq0}a_{i,j,k}(L_{0}-m)^{i} M_{0}^{j}\bigg[Y_{0}^{k}g_m(L_0,M_0,Y_0) -\frac{mk}{2}Y_{0}^{k-1}p_m(L_0,M_0,Y_0)\\
&+\frac{k(k-1)m^{2}}{4}Y_{0}^{k-2}a_m(L_0,M_0,Y_0) \bigg]\\
=&u(L_0-m,M_0,Y_0)g_m(L_0,M_0,Y_0)-\frac{m}{2}\frac{\partial u(L_0-m,M_0,Y_0)}{\partial Y_0}p_m(L_0,M_0,Y_0)\\
&+\frac{m^{2}}{4}\frac{\partial^2 u(L_0-m,M_0,Y_0)}{\partial Y_0^2}a_m(L_0,M_0,Y_0).
\end{aligned}
\end{eqnarray*}

By formulas (\ref{proYL})-(\ref{proYY}) in Proposition \ref{4}, it is easy to deduce that
\begin{eqnarray*}\label{2}
\begin{aligned}
&Y_m\cdot u(L_0,M_0,Y_0)=Y_m\cdot \sum\limits_{i,j,k\geq0}a_{i,j,k} L_{0}^{i}M_{0}^{j}Y_{0}^{k}\\
=& \sum\limits_{i,j,k\geq0}a_{i,j,k}(L_{0}-m)^{i} M_{0}^{j}\bigg[Y_{0}^{k}p_m(L_0,M_0,Y_0)-mkY_{0}^{k-1}a_m(L_0,M_0,Y_0)\bigg]\\
=&u(L_0-m,M_0,Y_0)p_m(L_0,M_0,Y_0)-m\frac{\partial u(L_0-m,M_0,Y_0)}{\partial Y_0}a_m(L_0,M_0,Y_0).
\end{aligned}
\end{eqnarray*}

Using formulas (\ref{proML})-(\ref{proMY}) in Proposition \ref{4}, one gets
\begin{eqnarray}\label{3}
&&\quad M_m\cdot u(L_0,M_0,Y_0)=M_m\cdot \sum\limits_{i,j,k\geq0}a_{i,j,k} L_{0}^{i}M_{0}^{j}Y_{0}^{k} \nonumber\\
&&= \sum\limits_{i,j,k\geq0}a_{i,j,k}(L_{0}-m)^{i} M_{0}^{j}Y_{0}^{k}a_m(L_0,M_0,Y_0)\nonumber\\
&&=u(L_0-m,M_0,Y_0)a_m(L_0,M_0,Y_0).
\end{eqnarray}

As a consequence, Lemma \ref{qd1} holds.
\end{proof}

\begin{lemm}\label{a0}
$a_m(L_0,M_0,Y_0)\neq 0$ for all $m\in\mathbb{Z}$.
\end{lemm}
\begin{proof}
Suppose that $a_m(L_0,M_0,Y_0)= 0$ for some $m\in\mathbb{Z}^*$, then $M_m\cdot M= 0$ owing to (\ref{3}).
Now according to relation (\ref{17}), one has
$$mM_0\cdot M=[L_{-m},M_m]\cdot M=0.$$
It suggests that $M_0\cdot M=0$, hence $a_0(L_0,M_0,Y_0)= 0$, which is a contradiction because of $a_0(L_0,M_0,Y_0)= M_0$.
Therefore, Lemma \ref{a0} holds.
\end{proof}

\begin{lemm}\label{lem1}
For all $m\in\mathbb{Z}$, we have $a_m(L_0,M_0,Y_0)\in\mathbb{C}[Y_0, M_0]$.
\end{lemm}
\begin{proof}
For $m\in\mathbb{Z}$, now assume that $a_m(L_0,M_0,Y_0)=\sum\limits_{i=0}^{k_m}b_{m,i}L_0^i$, in which $b_{m,i}\in\mathbb{C}[Y_0,M_0]$, $k_m\in\mathbb{N}$
and $b_{m,k_m}\neq0$. Due to relation (\ref{19}) and Proposition \ref{4}, one obtains the following
\begin{eqnarray*}
\begin{aligned}
0&={[M_n,M_m]}\cdot1\\
               &=\sum\limits_{i=0}^{k_m}b_{m,i}(L_0-n)^i\sum\limits_{j=0}^{k_n}b_{n,j}L_0^j
                 -\sum\limits_{j=0}^{k_n}b_{n,j}(L_0-m)^j\sum\limits_{i=0}^{k_m}b_{m,i}L_0^i\\
               &\equiv b_{m,k_m}b_{n,k_n}(mk_n-nk_m)L_0^{k_m+k_n-1}
               \qquad(\!\!\!\!\!\!\mod\mathop{\oplus}\limits_{i=0}^{k_m+k_n-2} \mathbb{C}[Y_0, M_0]L_0^i),
\end{aligned}
\end{eqnarray*}
yielding that $mk_n-nk_m=0$, thus $k_m=mk_1$.
If $k_1>0$, then it is nonsense for $k_m$ when $m$ is negative, which forces $k_1=0$, hence $k_m=0$ for $m\in\mathbb{Z}$, therefore $a_m(L_0,M_0,Y_0)\in\mathbb{C}[Y_0, M_0]$.
\end{proof}
\begin{lemm}\label{am}
For $m\in\mathbb{Z}$, one has $p_m(L_0,M_0,Y_0)\in\mathbb{C}[Y_0, M_0]$ and $a_m(L_0,M_0,Y_0)\in\mathbb{C}[M_0]$.
\end{lemm}
\begin{proof}
For all $m\in\mathbb{Z}$, according to Lemma \ref{lem1} assume that $a_m(L_0,M_0,Y_0)=\sum\limits_{i=0}^{h_m}l_{m,i}Y_0^i$,
where $l_{m,i}\in\mathbb{C}[M_0]$, $h_m\in\mathbb{N}$, $l_{m,h_m}\neq0$.
Moreover, let $p_m(L_0,M_0,Y_0)=\sum\limits_{i=0}^{t_m}d_{m,i}L_0^i$,
in which $d_{m,i}\in\mathbb{C}[Y_0, M_0]$, $t_m\in\mathbb{N}$ and $d_{m,t_m}\neq0$.
Applying relation (\ref{19}) and Proposition \ref{4}, one has
\begin{eqnarray}\label{YM1}
0&=&{[Y_n,M_m]}\cdot1\nonumber\\
                 &=&\sum\limits_{i=0}^{h_m}l_{m,i}\Big(Y_0^i\sum\limits_{j=0}^{t_n}d_{n,j}L_0^j-n iY_0^{i-1}\sum\limits_{j=0}^{h_n}l_{n,j}Y_0^j\Big)\nonumber\\
                  &&-\sum\limits_{j=0}^{t_n}d_{n,j}(L_0-m)^j\sum\limits_{i=0}^{h_m}l_{m,i}Y_0^i\nonumber\\
                 &=&\sum\limits_{i=0}^{h_m}\sum\limits_{j=0}^{t_n}l_{m,i}d_{n,j}L_0^j Y_0^i
                  -\sum\limits_{i=1}^{h_m}\sum\limits_{j=0}^{h_n}l_{m,i}l_{n,j}n iY_0^{i-1}Y_0^j\nonumber\\
                  &&-\sum\limits_{i=0}^{h_m}\sum\limits_{j=0}^{t_n}l_{m,i}d_{n,j}(L_0-m)^jY_0^i.
\end{eqnarray}
Observing all terms of $L_0$ in
(\ref{YM1}), one knows that
$t_n\leq 1$, that is $t_n=0$ or $1$.

We claim that $t_n=0$ for all $n\in\mathbb{Z}$. In fact, suppose $t_k=1$ for $k\in\mathbb{Z}^*$, then (\ref{YM1}) becomes
\begin{eqnarray*}
\begin{aligned}
0 =&m\sum\limits_{i=0}^{h_m}l_{m,i}d_{k,1}Y_0^i
-k\sum\limits_{i=1}^{h_m}\sum\limits_{j=0}^{h_k}l_{m,i}l_{k,j}iY_0^{i-1}Y_0^j,
\end{aligned}
\end{eqnarray*}
taking $m=k$ and $m=-k$ respectively, one has
\begin{eqnarray*}
\begin{aligned}
0=&k\sum\limits_{i=0}^{h_k}l_{k,i}d_{k,1}Y_0^i
   -k\sum\limits_{i=1}^{h_k}\sum\limits_{j=0}^{h_k}l_{k,i}l_{k,j}iY_0^{i-1}Y_0^j\\
 =&k\sum\limits_{i=0}^{h_k}l_{k,i}d_{k,1}Y_0^i
   -k\sum\limits_{j=1}^{h_k}\sum\limits_{i=0}^{h_k}l_{k,j}l_{k,i}jY_0^{j-1}Y_0^i\\
 =&k\sum\limits_{i=0}^{h_k}l_{k,i}Y_0^i \bigg(d_{k,1}-\sum\limits_{j=1}^{h_k}l_{k,j}jY_0^{j-1}\bigg)
\end{aligned}
\end{eqnarray*}
and
\begin{eqnarray*}
\begin{aligned}
0=&-k\sum\limits_{i=0}^{h_{-k}}l_{-k,i}d_{k,1}Y_0^i
  -k\sum\limits_{i=1}^{h_{-k}}\sum\limits_{j=0}^{h_k}l_{-k,i}l_{k,j}iY_0^{i-1}Y_0^j\\
 =&-k\sum\limits_{i=0}^{h_{-k}}l_{-k,i}\bigg(d_{k,1}Y_0
  +\sum\limits_{j=0}^{h_k}l_{k,j}iY_0^j\bigg)Y_0^{i-1},
\end{aligned}
\end{eqnarray*}
thus $d_{k,1}=\sum\limits_{j=1}^{h_k}l_{k,j}jY_0^{j-1}$
and $d_{k,1}Y_0+\sum\limits_{j=0}^{h_k}l_{k,j}h_{-k}Y_0^j=0$ which suggests $l_{k,0}h_{-k}=0$,
thereby $d_{k,1}=\sum\limits_{j=1}^{h_k}l_{k,j}jY_0^{j-1}=-\sum\limits_{j=1}^{h_k}l_{k,j}h_{-k}Y_0^{j-1}\neq0$,
which is impossible.
Therefore, we get $t_n=0$ for $n\in\mathbb{Z}$.

In addition, for $n\neq 0$ and $t_n=0$, it is easy to find that (\ref{YM1}) becomes
\begin{eqnarray*}
\begin{aligned}
0 =-n\sum\limits_{i=0}^{h_m}\sum\limits_{j=0}^{h_n}l_{m,i}l_{n,j}iY_0^{i-1}Y_0^j,
\end{aligned}
\end{eqnarray*}
whence $h_m=0$ for all $m\in\mathbb{Z}$, then $a_m(L_0,M_0,Y_0)\in\mathbb{C}[M_0]$.
\end{proof}

\begin{lemm}\label{lem6}
$deg_{Y_0}p_m(L_0,M_0,Y_0)=1$ for all $m\in\mathbb{Z}$.
\end{lemm}
\begin{proof}
For $m\in\mathbb{Z}$ assume that $p_m(L_0,M_0,Y_0)=\sum\limits_{i=0}^{f_m}q_{m,i}Y_0^i$,
$a_m(L_0,M_0,Y_0)=\sum\limits_{i=0}^{e_m}s_{m,i}M_0^i$,
in which $q_{m,i}\in\mathbb{C}[M_0],s_{m,i}\in\mathbb{C}$ and $q_{m,f_m}\neq0,s_{m,e_m}\neq0$, $f_m\in\mathbb{N}$, $e_m\in\mathbb{N}$.
Due to relation (\ref{18}), one easily finds
\begin{eqnarray}\label{9}
\begin{aligned}
{[Y_n,Y_m]}\cdot1=(m-n)M_{m+n}\cdot1.\\
\end{aligned}
\end{eqnarray}
By straightforward calculations, the left hand side of (\ref{9}) is
\begin{eqnarray*}
\begin{aligned}
&Y_nY_m\cdot1-Y_mY_n\cdot1\\
    =&\sum\limits_{i=0}^{f_m}q_{m,i}\left(Y_0^iY_n\cdot1-n iY_0^{i-1}M_n\cdot1\right)
     -\sum\limits_{j=0}^{f_n}q_{n,j}\left(Y_0^jY_m\cdot1-m jY_0^{j-1}M_m\cdot1\right)\\
    =&\sum\limits_{i=0}^{f_m}q_{m,i}\left(Y_0^i\sum\limits_{j=0}^{f_n}q_{n,j}Y_0^j-n iY_0^{i-1}\sum\limits_{k=0}^{e_n}s_{n,k}M_0^k\right)\\
     &-\sum\limits_{j=0}^{f_n}q_{n,j}\left(Y_0^j\sum\limits_{i=0}^{f_m}q_{m,i}Y_0^i-mjY_0^{j-1}\sum\limits_{k=0}^{e_m}s_{m,k}M_0^k\right)\\
    =&\sum\limits_{j=0}^{f_n}\sum\limits_{k=0}^{e_m}s_{m,k}q_{n,j}m jY_0^{j-1}M_0^k-\sum\limits_{i=0}^{f_m}\sum\limits_{k=0}^{e_n}s_{n,k}q_{m,i} n iY_0^{i-1}M_0^k .
\end{aligned}
\end{eqnarray*}

While the right hand side of (\ref{9}) does not contain the terms of $Y_0$ and is not zero from Lemma \ref{am} and Lemma \ref{a0},
thus we have $f_n=f_m\neq0$,
thereby one directly obtains $f_m=f_0=1$.
\end{proof}

\begin{lemm}\label{lem5}
For all $m\in\mathbb{Z}$, one has $deg_{L_0}g_m(L_0,M_0,Y_0)=1$ and $g_m(L_0,M_0,Y_0)\in\mathbb{C}[L_0,M_0]$.
\end{lemm}
\begin{proof}
Let $g_n(L_0,M_0,Y_0)=\sum\limits_{i=0}^{r_n}c_{n,i}L_0^i$,
in which $c_{n,i}\in\mathbb{C}[M_0, Y_0]$ and $c_{n,r_n}\neq0$, $r_n\in\mathbb{N}$.
Thanks to Lemma \ref{am}, for $m\in\mathbb{Z}$ assume that
$a_m(L_0,M_0,Y_0)=\sum\limits_{i=0}^{e_m}s_{m,i}M_0^i$, where $s_{m,i}\in\mathbb{C}$ and $s_{m,e_m}\neq0$, $e_m\in\mathbb{N}$.
Owing to relation (\ref{17}), it is easy to see that
\begin{eqnarray}\label{8}
\begin{aligned}
{[L_n,M_m]}\cdot1=mM_{m+n}\cdot1.\\
\end{aligned}
\end{eqnarray}

Then the left hand side of (\ref{8}) directly becomes
\begin{eqnarray}\label{10}
&&L_nM_m\cdot1-M_mL_n\cdot1\nonumber\\
    &=&\sum\limits_{i=0}^{e_m}s_{m,i}M_0^i\sum\limits_{j=0}^{r_n}c_{n,j}L_0^j
      -\sum\limits_{j=0}^{r_n}c_{n,j}(L_0-m)^j\sum\limits_{i=0}^{e_m}s_{m,i}M_0^i\nonumber\\
    &=&-\sum\limits_{i=0}^{e_m}\sum\limits_{j=1}^{r_n}s_{m,i}c_{n,j}M_0^i
      \left(-jmL_0^{j-1}+C_j^2(-m)^2L_0^{j-2}+\cdots+C_j^j(-m)^j\right).
\end{eqnarray}

Note that the right hand side of (\ref{8}) does not contain the terms of $L_0$, thereby one deduces that $r_n\leq 1$.
If $r_k=0$ for some $k\in\mathbb{Z}^*$, then the left hand side of (\ref{8}) vanishes, however the right hand side of (\ref{8}) is not equal to zero if $m\neq0$ as a result of Lemma \ref{a0}, hence $r_n=1$ for $n\in\mathbb{Z}$.

In addition, it is clearly that the right hand side of (\ref{8}) does not contain the terms of $Y_0$ according to Lemma \ref{am}, hence from (\ref{10}) one knows $c_{n,1}\in\mathbb{C}[M_0]$.

Next suppose that $g_n(L_0,M_0,Y_0)=c_{n,1}L_0+\sum\limits_{i=0}^{\lambda_n}\alpha_{n,i}Y_0^i$,
in which $c_{n,1}\neq 0$ and $\alpha_{n,i}\in\mathbb{C}[M_0]$, $\alpha_{n,\lambda_n}\neq0$, $\lambda_n\in\mathbb{N}$.
On the one hand,
\begin{eqnarray}\label{l1}
&&L_nL_m\cdot1-L_mL_n\cdot1\nonumber\\
    &=&L_n(c_{m,1}L_0+\sum\limits_{i=0}^{\lambda_m}\alpha_{m,i}Y_0^i)-L_m(c_{n,1}L_0+\sum\limits_{i=0}^{\lambda_n}\alpha_{n,i}Y_0^i)\nonumber\\
    &=&-\sum\limits_{i=0}^{\lambda_n}\alpha_{n,i}(Y_{0}^{i}L_{m}\cdot1-\frac{m i}{2}Y_{0}^{i-1}Y_{m}\cdot1+\frac{m^{2}i(i-1)}{4}Y_{0}^{i-2}M_{m}\cdot1)\nonumber\\
    &&+\sum\limits_{i=0}^{\lambda_m}\alpha_{m,i}(Y_0^iL_n\cdot1-\frac{n i}{2}Y_{0}^{i-1}Y_{n}\cdot1+\frac{n^{2}i(i-1)}{4}Y_{0}^{i-2}M_{n}\cdot1)\nonumber\\
    &&+c_{m,1}(L_0-n)L_n\cdot1-c_{n,1}(L_0-m)L_m\cdot1\nonumber\\
    &=&mc_{n,1}\sum\limits_{i=0}^{\lambda_m}\alpha_{m,i}Y_0^i-nc_{m,1} \sum\limits_{i=0}^{\lambda_n}\alpha_{n,i}Y_0^i
      -\sum\limits_{i=0}^{\lambda_m}\alpha_{m,i}\frac{n i}{2}Y_{0}^{i-1}Y_{n}\cdot1\nonumber\\
    &&+\sum\limits_{i=0}^{\lambda_n}\alpha_{n,i}\frac{m i}{2}Y_{0}^{i-1}Y_{m}\cdot1+\sum\limits_{i=0}^{\lambda_m}\alpha_{m,i}\frac{n^{2}i(i-1)}{4}Y_{0}^{i-2}M_{n}\cdot1 \nonumber\\ &&-\sum\limits_{i=0}^{\lambda_n}\alpha_{n,i}\frac{m^{2}i(i-1)}{4}Y_{0}^{i-2}M_{m}\cdot1+(m-n)c_{m,1}c_{n,1}L_0\nonumber,
\end{eqnarray}
on the other hand,
\begin{eqnarray}\label{r1}
\begin{aligned}
&(m-n)L_{m+n}\cdot1\\
=&(m-n)(c_{{m+n},1}L_0+\sum\limits_{i=0}^{\lambda_{m+n}}\alpha_{{m+n},i}Y_0^i).\nonumber
\end{aligned}
\end{eqnarray}
According to relation (\ref{15}), it is obvious that $\lambda_m=\lambda_n=\lambda_{m+n}=\lambda_0=0$, that is, $g_m(L_0,M_0,Y_0)\in\mathbb{C}[L_0,M_0]$.
\end{proof}

From what has been discussed above, now we turn to the proof of Theorem \ref{th}.

\noindent\textbf{The Proof of Theorem \ref{th}.}

Lemma \ref{qd1} suggests that $g_m(L_0,M_0,Y_0), a_m(L_0,M_0,Y_0)$ and $p_m(L_0,M_0,Y_0)$ completely determine the action of $L_m, M_m$ and $Y_m$ on $M$. Thus the next task is just to determine
$g_m(L_0,M_0,Y_0),$ $a_m(L_0,M_0,Y_0),$ and $p_m(L_0,M_0,Y_0)$ for all $m\in\mathbb{Z}$.

According to Lemma \ref{am}, Lemma \ref{lem5} and Lemma \ref{lem6},
here we need to rewrite the preceding assumptions. For $m\in\mathbb{Z}$, suppose that
\begin{eqnarray}
g_m(L_0,M_0,Y_0)=a_m+b_mL_0,\label{gm}\\
p_m(L_0,M_0,Y_0)=c_m+d_mY_0\label{pm},
\end{eqnarray}
in which $a_m,b_m,c_m,d_m\in\mathbb{C}[M_0]$ and $b_m\neq0, d_m\neq0$. In particular, $a_0=c_0=0$ and $b_0=d_0=1$.
Thanks to relation (\ref{15}), one gets directly
\begin{eqnarray}\label{LL1}
\begin{aligned}
{[L_n,L_m]}\cdot1=(m-n)L_{m+n}\cdot1.\\
\end{aligned}
\end{eqnarray}
Substituting (\ref{gm}) into (\ref{LL1}), one immediately finds that the left hand side of (\ref{LL1}) becomes
\begin{eqnarray}\label{l1}
&&L_nL_m\cdot1-L_mL_n\cdot1\nonumber\\
    &=&L_n(a_m+b_mL_0)-L_m(a_n+b_nL_0)\nonumber\\
    &=&ma_mb_n-na_nb_m+(m-n)b_nb_mL_0
\end{eqnarray}
and the right hand side of (\ref{LL1}) becomes
\begin{eqnarray}\label{r1}
\begin{aligned}
&(m-n)L_{m+n}\cdot1\\
=&(m-n)(a_{m+n}+b_{m+n}L_0).
\end{aligned}
\end{eqnarray}
Comparing the coefficients of $L_0$ of (\ref{l1}) and (\ref{r1}), one has $b_nb_m=b_{m+n}$, yielding that $b_m=\lambda^m(\lambda:=b_1\in\mathbb{C}^*)$ for all $m\in\mathbb{Z}$.
Moreover, observing the constant term in (\ref{l1}) and (\ref{r1}), one obtains that $ma_mb_n-na_nb_m=(m-n)a_{m+n}$.
Choosing $m=1$ in the formula gives $\lambda^na_1-n\lambda a_n=(1-n)a_{n+1}$, from which we have
$a_m=m\lambda^{m}\alpha$ by induction on $n$, where $\alpha=\frac{a_1}{\lambda}$ and $a_1\in\mathbb{C}$.
Therefore, $g_m(L_0,M_0,Y_0)=\lambda^m(L_0+m\alpha)$.

Similarly, according to relation (\ref{16}), it is effortless to see that
\begin{eqnarray}\label{LY1}
\begin{aligned}
{[L_n,Y_m]}\cdot1=(m-\frac{n}{2})Y_{m+n}\cdot1.\\
\end{aligned}
\end{eqnarray}
Thanks to (\ref{gm}) and (\ref{pm}) it is easy to deduce that the left hand side of (\ref{LY1}) is
\begin{eqnarray}
&&L_nY_m\cdot1-Y_mL_n\cdot1\nonumber\\
    &=&L_n(c_m+d_mY_0)-Y_m(a_n+b_nL_0)\nonumber\\
    &=&mc_mb_n-\frac{n}{2}d_mc_n+mb_nd_mY_0-\frac{n}{2}d_md_nY_0
\end{eqnarray}
and applying (\ref{pm}) one finds that the right hand side of (\ref{LY1}) becomes
\begin{eqnarray}
&&(m-\frac{n}{2})Y_{m+n}\cdot1\nonumber\\
&=&(m-\frac{n}{2})(c_{m+n}+d_{m+n}Y_0).
\end{eqnarray}
Hence $mc_mb_n-\frac{n}{2}d_mc_n=(m-\frac{n}{2})c_{m+n}$ and $mb_nd_m-\frac{n}{2}d_md_n=(m-\frac{n}{2})d_{m+n}$.
Taking $m=1$ in the latter formula one gets $\lambda^nd_1-\frac{n}{2}d_1d_n=(1-\frac{n}{2})d_{n+1}$, which implies that $d_m=\lambda^m$ for $m\in\mathbb{Z}$.
Besides, choosing $m=\pm 1$ in the former formula gives
\begin{eqnarray}\label{m=1}
\begin{aligned}
\lambda^nc_1-\frac{n}{2}\lambda c_n=(1-\frac{n}{2})c_{n+1}
\end{aligned}
\end{eqnarray}
and
\begin{eqnarray}\label{m=-1}
\begin{aligned}
-c_{-1}\lambda^n-\frac{n}{2}\lambda^{-1}c_n=(-1-\frac{n}{2})c_{n-1}.
\end{aligned}
\end{eqnarray}
Setting $n=\pm 1$ in (\ref{m=1}), one deduces that
$c_{2}=\lambda c_1$
and
$c_{-1}=-2\lambda^{-2}c_1$.
In addition, taking $n=2$ in (\ref{m=-1}), one easily has
$$-c_{-1}\lambda^2-\lambda^{-1}c_2=-2c_{1}.$$
It yields that
$c_2=4\lambda c_1$,
whence $c_1=0$.
Thus from (\ref{m=1}) or (\ref{m=-1}), it is easy to find that $c_m=0$ for $m\in\mathbb{Z}$.
Therefore $p_m(L_0,M_0,Y_0)=\lambda^m Y_0$.

Now it remains to consider how to determine $a_m(L_0,M_0,Y_0)$.
For $m\neq 0$, continue to use the assumption $a_m(L_0,M_0,Y_0)=\sum\limits_{i=0}^{e_m}s_{m,i}M_0^i$,
in which $s_{m,i}\in\mathbb{C}$ and $s_{m,e_m}\neq0$.
Due to relation (\ref{17}), it is clearly to see that
\begin{eqnarray}\label{LM1}
\begin{aligned}
\ [L_{n},M_{m}]\cdot1=mM_{m+n}\cdot1.
\end{aligned}
\end{eqnarray}
Now calculating the left hand side of (\ref{LM1}), one always has
\begin{eqnarray*}
\begin{aligned}
&L_nM_m\cdot1-M_mL_n\cdot1\\
    =&\sum\limits_{i=0}^{e_m}s_{m,i}\lambda^n(L_0+n\alpha)M_0^i
     -\sum\limits_{i=0}^{e_m}s_{m,i}\lambda^n(L_0-m+n\alpha)M_0^i\\
    =&m\sum\limits_{i=0}^{e_m}s_{m,i}\lambda^nM_0^i,\\
\end{aligned}
\end{eqnarray*}
thus one immediately finds that
\begin{eqnarray}\label{LM2}
\begin{aligned}
m\sum\limits_{i=0}^{e_m}s_{m,i}\lambda^nM_0^i=m\sum\limits_{i=0}^{e_{m+n}}s_{m+n,i}M_0^i.
\end{aligned}
\end{eqnarray}
Whence one knows $e_m =e_{m+n}=e_0=1$ for all $m\in\mathbb{Z}$.
Therefore (\ref{LM2}) becomes
$$m\lambda^n(s_{m,0}+s_{m,1}M_0)=m(s_{m+n,0}+s_{m+n,1}M_0).$$
Hence one always has $s_{m+n,0} =\lambda^ns_{m,0}$ and $s_{m+n,1}=\lambda^ns_{m,1}$,
from which it is direct to obtain that $s_{m,0}=\lambda^ms_{0,0}=0$ and $s_{m,1}=\lambda^ms_{0,1}=\lambda^m$ for all $m\in\mathbb{Z}$.

By putting everything together, it can be seen that
\begin{eqnarray*}
\begin{aligned}
g_m(L_0,M_0,Y_0)&=\lambda^m (L_0+m\alpha),\\
p_m(L_0,M_0,Y_0)&=\lambda^m Y_0,\\
a_m(L_0,M_0,Y_0)&=\lambda^m M_0,
\end{aligned}
\end{eqnarray*}
for $\lambda\in\mathbb{C}^*$, $\alpha\in\mathbb{C}$ and all $m\in\mathbb{Z}$.
Thus according to Lemma \ref{qd1}, Theorem \ref{th} holds.
\qed

\section{\textbf{Modules over $\mathfrak{sv}(\frac{1}{2})$}}
In this section, we focus on the Schr\"{o}dinger-Virasoro algebra $\mathfrak{sv}(s)$ for $s=\frac{1}{2}$ and show that the class of free $U(\mathbb{C}L_0\oplus \mathbb{C}M_0)$-modules of rank $1$ over the algebra $\mathfrak{sv}(\frac{1}{2})$ is nonexistent.

The Schr\"{o}dinger-Virasoro  algebra $\mathfrak{sv}(\frac{1}{2})$ is an infinite-dimensional Lie algebra over $\mathbb{C}$ with a basis
$$\{L_{n},M_{n},Y_{p}|n \in \mathbb{Z}, p\in \mathbb{Z}+\frac{1}{2}\}$$
and has the following relations
\begin{eqnarray}
&&[L_{n},L_{m}]=(m-n)L_{m+n},\label{sLL}\\
&&[L_{n},Y_{p}]=(p-\frac{n}{2})Y_{p+n},\label{sLY}\\
&&[L_{n},M_{m}]=mM_{m+n},\label{sLM}\\
&&[Y_{p},Y_{q}]=(q-p)M_{p+q},\label{sYY}\\
&&[Y_{p},M_{m}]=[M_{n},M_{m}]=0,\label{sYM}
\end{eqnarray}
where $m, n\in\mathbb{Z}, p,q\in \mathbb{Z}+\frac{1}{2}$.

From the above relations (\ref{sLL})-(\ref{sYM}), there are some useful formulas that are similar to Proposition \ref{4}.
\begin{prop}\label{pro2}
For $i\in\mathbb{Z}, n\in\mathbb{Z}, p\in \mathbb{Z}+\frac{1}{2}$ and $i\geq 0$, the following formulas hold:
\begin{enumerate}
  \item $M_{n}L_{0}^{i}=(L_{0}-n)^{i}M_{n}$.\label{pro2ML}
  \item $M_{n}M_{0}^{i}=M_{0}^{i}M_{n}$.\label{pro2MM}
  \item $Y_{p}L_{0}^{i}=(L_{0}-p)^{i}Y_{p}$.\label{pro2YL}
  \item $Y_{p}M_{0}^{i}=M_{0}^{i}Y_{p}$.\label{pro2YM}
  \item $L_{n}L_{0}^{i}=(L_{0}-n)^{i}L_{n}$.\label{pro2LL}
  \item $L_{n}M_{0}^{i}=M_{0}^{i}L_{n}$.\label{pro2LM}
\end{enumerate}
\end{prop}

It is straightforward to check these formulas by induction on $i$ for all $i\geq 0$.
\begin{theorem}\label{th2}
Let $\mathfrak{sv}(\frac{1}{2})$ be the Schr\"{o}dinger-Virasoro algebra $\mathfrak{sv}(s)$ for $s=\frac{1}{2}$, denote by $\mathcal{M}$ the set of all free $U(\mathbb{C}L_0\oplus \mathbb{C}M_0)$-modules of rank $1$ over $\mathfrak{sv}(\frac{1}{2})$, then $$\mathcal{M}=\varnothing.$$
\end{theorem}
\begin{proof}
Assume on the contrary that $\mathcal{M}\neq\varnothing$ and let $N\in\mathcal{M}$, so $N=U(\mathbb{C}L_0\oplus \mathbb{C}M_0)$.
Take any $w(L_0,M_0)=\sum\limits_{i,j\geq0}a_{i,j}L_{0}^{i}M_{0}^{j}\in N$ and
suppose that
\begin{eqnarray*}
\begin{aligned}
&L_m\cdot1=g_m(L_0,M_0),\\
&M_m\cdot1=a_m(L_0,M_0),\\
&Y_p\cdot1=h_p(L_0,M_0),
\end{aligned}
\end{eqnarray*}
where $m\in\mathbb{Z}$, $p\in \mathbb{Z}+\frac{1}{2}$ and $g_m(L_0,M_0),a_m(L_0,M_0)$, $h_p(L_0,M_0)\in N$.

Next we will prove that $g_m(L_0,M_0), a_m(L_0,M_0)$ and $h_p(L_0,M_0)$ completely determine the action of $L_m, M_m$ and $Y_p$ on $N$.

It follows from formulas (\ref{pro2LL})-(\ref{pro2LM}) in Proposition\ref{pro2}
\begin{eqnarray}
L_m\cdot w(L_0,M_0)&=& L_m\cdot \left(\sum\limits_{i,j\geq0}a_{i,j}L_{0}^{i}M_{0}^{j}\right) \nonumber\\
                       &=& \sum\limits_{i,j,k\geq0}a_{i,j,k}(L_{0}-m)^{i} M_{0}^{j}g_m(L_0,M_0) \nonumber\\
                       &=& w(L_0-m,M_0)g_m(L_0,M_0).\label{Lw}
\end{eqnarray}
According to formulas (\ref{pro2YL})-(\ref{pro2YM}) in Proposition\ref{pro2}, one always has
\begin{eqnarray}
Y_p\cdot w(L_0,M_0)&=& Y_p\cdot \left(\sum\limits_{i,j\geq0}a_{i,j}L_{0}^{i}M_{0}^{j}\right) \nonumber\\
                       &=& \sum\limits_{i,j\geq0}a_{i,j}(L_{0}-p)^{i} M_{0}^{j}h_p(L_0,M_0) \nonumber\\
                       &=& w(L_0-p,M_0)h_p(L_0,M_0).\label{Yw}
\end{eqnarray}
By formulas (\ref{pro2ML})-(\ref{pro2MM}) in Proposition\ref{pro2}, it is easy to obtain that
\begin{eqnarray}
M_m\cdot w(L_0,M_0)&=& M_m\cdot \left(\sum\limits_{i,j\geq0}a_{i,j}L_{0}^{i}M_{0}^{j}\right) \nonumber\\
                       &=& \sum\limits_{i,j\geq0}a_{i,j}(L_{0}-m)^{i} M_{0}^{j}a_m(L_0,M_0) \nonumber\\
                       &=& w(L_0-m,M_0)a_m(L_0,M_0).\label{Mw}
\end{eqnarray}

The above three formulas (\ref{Lw})-(\ref{Mw}) suggest that one only needs to consider how to determine
$g_m(L_0,M_0),$ $a_m(L_0,M_0),$ and $h_p(L_0,M_0)$ for $m\in\mathbb{Z}$ and $p\in \mathbb{Z}+\frac{1}{2}$ if one wants to know the action of $L_m, M_m$ and $Y_p$ on $M$. Whence we will discuss how to determine them in the next step.

For $m\in\mathbb{Z}$, assume that $a_m(L_0,M_0)=\sum\limits_{i=0}^{k'_m}b'_{m,i}L_0^i$, in which $b'_{m,i}\in\mathbb{C}[M_0]$, $k'_m\in\mathbb{N}$
and $b'_{m,k'_m}\neq0$. According to relation (\ref{sYM}) and Proposition \ref{pro2}, one obtains $k'_m=0$ referring to the proof of Lemma \ref{lem1},
therefore $a_m(L_0,M_0)\in\mathbb{C}[M_0]$.

In addition, we claim that $a_m(L_0,M_0)\neq 0$ for all $m\in\mathbb{Z}$.
In fact, if there exists $m\in\mathbb{Z}^*$ such that $a_m(L_0,M_0)= 0$, then $M_m\cdot M= 0$ due to formula (\ref{Mw}). According to relation (\ref{sLM}), one immediately deduces that
$$mM_0\cdot M=[L_{-m},M_{m}]\cdot M=0,$$
which implies that $a_0(L_0,M_0)= 0$. This is a contradiction.

For $m\in\mathbb{Z}^*$ and $p\in \mathbb{Z}+\frac{1}{2}$, suppose that $a_m(L_0,M_0)=\sum\limits_{i=0}^{e'_m}s'_{m,i}M_0^i$ and $h_p(L_0,M_0)=\sum\limits_{i=0}^{t'_p}d'_{p,i}L_0^i$,
in which $s'_{m,i}\in\mathbb{C}$, $e'_m\in\mathbb{N}$, $s'_{m,e'_m}\neq0$ and $d'_{p,i}\in\mathbb{C}[M_0]$, $t'_p\in\mathbb{N}$, $d'_{p,t'_p}\neq0$.
Combining relation (\ref{sYM}) with Proposition \ref{pro2}, one has
\begin{eqnarray*}\label{YM}
\begin{aligned}
0=&{[Y_p,M_m]}\cdot1\\
                 =&\sum\limits_{i=0}^{e'_m}s'_{m,i}M_0^i\sum\limits_{j=0}^{t'_p}d'_{p,j}L_0^j
                  -\sum\limits_{j=0}^{t'_p}d'_{p,j}(L_0-m)^j\sum\limits_{i=0}^{e'_m}s'_{m,i}M_0^i\\
                  =&-\sum\limits_{i=0}^{e'_m}\sum\limits_{j=0}^{t'_p}s'_{m,i}d'_{p,j}M_0^i
\left(-jmL_0^{j-1}+C_j^2(-m)^2L_0^{j-2}+\cdots+C_j^j(-m)^j\right),
\end{aligned}
\end{eqnarray*}
hence $t'_p=0$, that is $h_p(L_0,M_0)\in\mathbb{C}[M_0]$.
Thus from (\ref{pro2YM}) in Proposition \ref{pro2} and relation (\ref{sYY}), we always have
$$[Y_{p},Y_{q}]\cdot 1=0\neq(q-p)M_{p+q}$$
for $p\neq q$, which is impossible.

From what has been discussed above, we draw a conclusion that the class of free $U(\mathbb{C}L_0\oplus \mathbb{C}M_0)$-modules of rank $1$ over the algebra $\mathfrak{sv}(\frac{1}{2})$ is not existent, which completes the proof of Theorem \ref{th2}.
\end{proof}

\vskip10pt \centerline{\bf ACKNOWLEDGMENT}

This work was supported by the NSFC (Grant No. 11871325 and 11726016).
\smallskip

\bibliographystyle{amsalpha}

\end{document}